\documentclass[a4paper,11pt]{elsarticle}

\usepackage{cite}
\usepackage{amsmath,amssymb,amsfonts,amsthm}
\usepackage{algorithmic}
\usepackage{graphicx}
\usepackage{textcomp}
\usepackage{xcolor}
\usepackage{color}
\newtheorem{myDef}{Definition}
\newtheorem{lemma}{Lemma}
\newtheorem{theorem}{Theorem}

\newtheorem{remark}{Remark}
\newtheorem{proposition}{Proposition}

 \usepackage{amssymb}

 \setcitestyle{authoryear,open={(},close={)}}

\usepackage[lofdepth,lotdepth]{subfig}

\begin{document}

\begin{frontmatter}
\title{\bf Sustainable cooperation on the hybrid pollution-control game with heterogeneous players}
\tnotetext[t1]{The work of the second author  was supported by the Russian Science Foundation under grant No. 24-21-00302, https://rscf.ru/en/project/24-21-00302/}

\author[1]{Yilun Wu\corref{cor1}} 
\ead{wuyilun310@gmail.com}
\author[1]{Anna Tur}
\ead{a.tur@spbu.ru}
\author[1]{Peichen Ye}
\ead{Yepeichen@outlook.com}
\cortext[cor1]{Corresponding author}
\affiliation[1]{organization={Saint Petersburg State University},
city={Saint Petersburg},
country={Russia}}

\begin{abstract}
This paper considers a hybrid pollution-control differential game with two farsighted players and one myopic player. Both the seasonal regime shifts in the state dynamics and the players' heterogeneous preferences are introduced into the model. The strategies under cooperative, noncooperative and partially cooperative scenarios are obtained by utilizing  the Pontryagin's Maximum Principle. Under all feasible coalition structures, the convergence of the state variable is proved. A new sustainably--cooperative optimality principle is proposed according to the coalition structures, which belongs to the imputation set.  The prerequisite for the existence of time-consistency in the sustainably-cooperative optimality principle is explicitly obtained. The seasonal imputation distribution procedure (IDP) is designed to maintain the time-consistentcy (dynamic stability) of cooperation over time.
\end{abstract}


\end{frontmatter}
\section{Introduction}
Environmental pollution and emission reduction problems, as the difficulties faced by human beings and regional development, have become one of the research focuses of the academic community \citep{Breton2005A, J2010,Dez2018}. 

When considering models of ecological management, it is necessary to take into account possible seasonal changes in the environment, which leads to the need to consider models with regime shifts. Phenomena of switching behavior in decision--making tasks are pervasive and ubiquitous across a wide range of real--world applications and domains. In the control theory domain, hybrid optimal control methods are mainly used to overcome decision--making problems with switching phenomena \citep{riedinger2003optimal, shaikh2007hybrid, bonneuil2016optimal, gromov2022periodic}. In \citep{Par2024}, a two-stage optimal control problem with state-driven switching is solved by employing the Pontryagin’s Maximum Principle. In \citep{Dez2024}, models in environmental economics are investigated with stochastic time-driven tipping moment, solution are derived through the Hamilton-Jacobi-Bellman equation.
 
Alternatively, considering complex switching mechanisms and players' interaction, in \citep{Gromov2017}, it has been demonstrated that hybrid differential games can be rigorously formulated drawn from the paradigm of hybrid systems, and that the problem of determining the strategies in the game-theoretic problems can be formally posed and solved using the methodologies of hybrid optimal control theory. In \citep{gromov2014differential}, hybrid differential games with random duration \citep{Jes2012Heterogeneous, De2012,de2013,gromov2014differential,Gromova2016,Gromova2021} and regime shifts \citep{Gromov2017} were considered in detail  within a hybrid optimal control framework.  Optimal pollution control in the dynamic system with periodic switching is first considered in \citep{gromov2022periodic}. Subsequently, the model is generalized to game-theoretic frameworks with regime shifts and random duration in \citep{Wu2023}.

Although global cooperation and coordination can benefit everyone involved, each country would like to unilaterally free ride on everyone else’s efforts, 
which  constitutes a significant obstacle to effective international collaboration. However, the International environmental agreements (IEAs) makes an effort to address transboundary environmental issues by maintaining the stability of cooperation \citep{Alt2013}. Therefore, the examination of the exogenous factors arising from different coalition structures aims at discussing the necessity of cooperation between nations or players in pollution control problems. On the one hand, it provides an alternative paradigm for modeling IEAs according to \citep{fin2008}. On the other hand, it explores and resolves possible conflicts of interest among players \citep{su2024}. In \citep{su2023, Fanokoa2011}, the condition for the stability of the grand coalition with two developed countries and one developing is investigated. However, the time--inconsistency issue was not considered. For mathematical tractability, unlike several papers in the literature (see, e.g., \citep{Fanokoa2011} or \citep{su2023}), where the cost function is quadratic in the pollution stock, this paper follows the classical setting where the cost function is linear with respect to the pollution stock. This setting not only simplifies the analysis on open-loop strategies and state dynamics, but also facilitates the derivation of the seasonal imputation distribution procedure (IDP).

In this paper, we consider a 3-player differential game, in which players face the exogenous effect -- environmental (e.g., climate) changes, meanwhile, endogenous effect -- heterogeneous preferences. This work endeavors to address and provide insights into the following research queries: 

1. With climate changes, what decision-making processes will be employed by heterogeneous players  under all feasible coalition structures? 

2. The significance of the grand coalition formation, i.e., can players achieve more in grand coalition? 

3. Whether dynamic stability (time-consistency) of cooperation can be achieved?

The main contributions of this work are summarized in the following highlights:

1. For the 3-player hybrid differential game with heterogeneous players, open-loop strategies under all feasible coalition structures are uniquely derived.

2. In each feasible coalition structure, we prove that the state variable with arbitrary initial values is convergent to a unique hybrid limit cycle.

3. Based on distinct mechanisms of coalition formation, a novel sustainably-cooperative optimality principle is proposed for the 3-player hybrid differential game, in which players have different preferences.

4. The condition that the time-consistency of the sustainably-cooperative optimality principle exists is obtained.

The findings of this study have significant applications on decision making processes in robotics (e.g., multi-robot coalition formation problem), economics (e.g., IEAs) and engineering (e.g., distributed communication system). The preferences exhibited by players in the game may be influenced by individual self-interests, such as minimizing/maximizing the instantaneous payoff (e.g., energy consumption, net profit, mission time or path tracking error), as well as social considerations, including the complementary nature of other players' capabilities and the avoidance of congestion/failure when accessing shared resources.

The rest of the paper is structured as follows. Section 2 formulated the model of the 3-player pollution control differential game with time-driven switches and heterogeneous (farsighted and myopic) preferences. Section 3 presents decision-making processes under all possible coalition structures, the cooperative, noncooperative and partially cooperative solutions are obtained explicitly. Section 4 investigates dynamics of the state under different coalition structures, the convergence of the state variable is proved. In Section 5, a novel sustainably-cooperative optimality principle is proposed and the prerequisite for the existence of its time-consistency is derived. The imputation distribution procedure (IDP) is designed  for ensuring time-consistency. Section 6 briefly concludes the paper.

\section{The model}
We consider a hybrid differential game  involving  three neighboring players, where each player (e.g., country or factory) engages in production work and simultaneously controls the policy of his/her pollution emission into the common environment of  a specific region. Given the fluidity and accumulation of regional pollution problems, the common region of the players is affected not only by the negative externalities of each player's own emissions, but also by the emissions of neighboring players. The pollution stock accumulates over time as players continue to emit. The pollution stock at time $t$ is $x(t)$. We assume that the pollution stock is non-negative for all $t>0$. Thus, the dynamics of pollution stock with regime shifts \citep{Wu2023} can be described as the following differential equation:

\begin{equation}\label{eq:$x$}
\dot{z}(t)=\sum_{i=1}^{3}\xi_i v_i(t)-\delta(t) z(t), \ z(0)=z_0, \\
\end{equation}
where $z_0>0$ is the initial pollution stock, $ v_i\in [0,b_i]$ denotes the emission rate of player $i$, $b_i$ is player $i's$ maximum emission rate, $\xi_i >0$ is the marginal influence on pollution accumulation, $\delta(t)>0$ represents the environment's ability to decompose pollution or self-purify rate, which is defined by a periodic piecewise constant function as follows:
\begin{equation}\label{eq:$beta$}
\delta(t):=\begin{cases}\delta_1 > 0,& t\in[kT,(k+ \tau)T),\\\delta_2> 0,& t\in[(k+\tau)T),(k+1)T)].\end{cases}
\end{equation}

 Accounting for the seasonal fluctuation, in this setting, the periodic switching characteristic of the environment's self-purify rate can be modeled. Each period $T$ is divided into two parts. For $k\in \mathbb{N}_0$ and $\tau \in (0,1)$, in the first part, i.e., $t\in [kT,(k+ \tau)T)$, the environment's self-purify rate is $\delta_1$; in the second part, i.e., $t\in [(k+\tau)T),(k+1)T)]$, the environment's self-purify rate is $\delta_2$.

In this work, the hybrid differential game of pollution control is considered with a myopic player and two farsighted players \citep{Fanokoa2011}. In particular, we suppose that the farsighted players are vulnerable to pollution and the myopic player is non-vulnerable to pollution. This setting can be employed for regional pollution control problems between developed and developing countries \citep{su2023}. 

Therefore, considering the importance that farsighted players place on the pollution stock, the payoff functional of the farsighted player $i=1,2$ to be maximized is given by
\begin{equation}\label{eq:$k1$}
K_i(z_0,v)=\int_{0}^{\infty} e^{-\rho t}[a_iv_i(b_i-v_i/2) - q_iz]\, dt, \\
\end{equation}
where $a_i>0$ converts the production stream into a profit stream, $q_i>0$ represents for the tax that player $i$ has to bear. Assume that the myopic player is indifferent to the stock of environmental pollution, the payoff functional of the myopic player $i=3$ to be maximized is given by
\begin{equation}\label{eq:$k2$}
K_3(z_0,v)=\int_{0}^{\infty} e^{-\rho t}[a_3v_3(b_3-v_3/2)]\, dt. \\
\end{equation}

\section{Players' strategies under different coalition structures}

In this section, we examine all possible coalition structures in this 3-player game. Emission policies of all players can be derived under cooperative, noncooperative and partial-cooperative scenarios. The corresponding possible coalition structures are given by: 1. Cooperative scenario: $\pi_1=\{1,2,3\}$; 2. Noncooperative scenario: $\pi_2=\{  \{1\},\{2\},\{3\}\}$; 3. Partial- cooperative scenarios: $\pi_3=\{  \{1,2\},\{3\}\} \ ,\pi_{41}=\{\{1,3\},\{2\} \}\ , \pi_{42}=\{ \{2,3\}, \{1\} \}$. The sustainability of cooperation and the conditions for its achievement are the main issues of concern in this work. 

\subsection{Cooperative scenario: $\pi_1=\{1,2,3\}$}

Under this scenario, players collaborate to maximize their joint payoff

\begin{equation}\label{eq:$kco$}
K^{\pi_1}=K_1+K_2+K_3=\sum_{i=1}^{3}\int_{0}^{\infty} e^{-\rho t} [\sum_{i=1}^{3}(a_iv_i(b_i-v_i/2)) - qz]\, dt, \\
\end{equation}
where $s=q_1+q_2$. 

Employing the Pontryagin's Maximum Principle, the current value Hamiltonian can be obtained as follows:
\begin{equation}\label{eq:$Hco$}
 H(z,v,\lambda)=\sum_{i=1}^{3}(a_iv_i(b_i-v_i/2)) - qz+ \lambda(\sum_{i=1}^{3}\xi_i v_i-\delta(t) z).
\end{equation}

Hence, the canonical system is obtained:
\begin{equation}\begin{aligned}\label{eq:$csN$}
&\dot{z}=\sum_{i=1}^{3}\xi_i v_i-\delta(t) z, \\
&\dot{\lambda}=\rho\lambda-\frac{\partial H}{\partial z}=s+(\rho+\delta(t))\lambda.
\end{aligned}\end{equation}

Through the first-order optimality condition of \eqref{eq:$Hco$}, the optimal control of player $i$ can be obtained as follows:
\begin{equation}\label{eq:$uco$}
v^{co}_{i}=\begin{cases}b_i,&\lambda(t)>0,\\ b_i+\frac{\xi_i}{a_i}\lambda(t),&\lambda(t) \in[-\frac{a_ib_i}{\xi_i},0],\\ 0,& \lambda(t)<-\frac{a_ib_i}{\xi_i}.\end{cases}
\end{equation}

\begin{myDef}\label{def1}
   Emission strategy $v_{i}(t)$ is {\bf environmentally sustainable}, see \citep{G1995},  if it does not take on boundary values except at isolated moments in time, i.e., $\lambda(t) \in[-\frac{a_ib_i}{\xi_i},0]$, for all $t>0$.
\end{myDef}

In this setting, for $\lambda(0)=\lambda(T)$, a unique optimal solution can be determined, which forms a hybrid limit cycle with period $T$, see \citep{gromov2022periodic}. 

Moreover, $\lambda(t)$ diverges if $\lambda(0) \neq \lambda(T)$, see \citep{Wu2023}. Then, after some $t^{\prime}$, player halts the production ($v^{co}_{i}(t)=0$) or player produces at maximum production rate ($v^{co}_{i}(t)=b_i$) for $t\geq t^{\prime}$. These situations are inconsistent with the concept of sustainable development. According to this background, we only examine the situation that when $v^{co}_{i}$ is environmentally sustainable.  

Define $L(t)$ as follows:

\begin{equation}\label{eq:$L_t$}
L(t)=\begin{cases}m_1e^{s_1(t-kT)}-\frac{1}{s_1},&t\in[kT,(k+\tau)T],\\ m_2e^{s_2(t-kT)}-\frac{1}{s_2},&t \in[(k+\tau)T,(k+1)T],\end{cases}
\end{equation}
where $s_1=\delta_1+\rho,\ s_2=\delta_2+\rho,\ m_1=\frac{(s_2-s_1)(1-e^{s_2 T(\tau-1)})}{s_1s_2(e^{s_1 \tau T}-e^{s_2 T(\tau-1)})}, \ m_2=e^{-s_2T}(m_1-\frac{1}{s_1}+\frac{1}{s_2})$.

Thus, through $\lambda(0)=\lambda(T)$, we obtain the limit cycle solution of the adjoint variable  $\lambda(t)=q\cdot L(t)$. 

Therefore, the equilibrium initial value  $\lambda(0)=\lambda_{hlc}$ can be uniquely derived as follows: 
\begin{equation}\label{eq:$lamhlc$}
\lambda_{hlc}=q(m_1-\frac{1}{s_1}).
\end{equation}



\subsection{Noncooperative scenario: $\pi_2=\{  \{1\},\{2\},\{3\}\}$}

In the absence of a cooperation agreement or failure to reach a cooperation agreement among any number of players. Players make decisions independently to maximize their individual payoff $K_i$, choosing the Nash equilibrium emission path under the constraint mechanism of pollution dynamics.

Player $3$, with a myopic perspective, tends to focus on short-term gains and neglect the long-term effects of their harm to the environment. The emission strategy of the player $3$ can be derived directly
\begin{equation}\label{eq:$u3NE$}
v_3^{NE}=b_3.
\end{equation}

Note that when lacking of consideration for environmental pollution, player $3$ sets the emission/production rate to its maximum value. This occurs because the player's objective is to maximize their own payoff, and does not consider the potential negative externalities of their actions on the environment.

For the farsighted players $j=1,2$, we determine their Nash equilibrium by utilizing the Pontryagin's Maximum Principle.

From their own integral payoff \eqref{eq:$k1$}, the current value Hamiltonian of the player $j$ is given by
\begin{equation}\label{eq:$HNE$}
 H_j(z,v,\lambda)=a_iv_i(b_i-v_i/2) - q_jz+ \lambda_j(\sum_{i=1}^{3}\xi_i v_i-\delta(t) z).
\end{equation}

Hence, the canonical system is obtained:
\begin{equation}\begin{aligned}\label{eq:$csNE$}
&\dot{z}=\sum_{i=1}^{3}\xi_i v_i-\delta(t) z, \\
&\dot{\lambda_j}=\rho\lambda_j-\frac{\partial H_j}{\partial z}=q_j+(\rho+\delta(t))\lambda_j.
\end{aligned}\end{equation}

To find the environmentally sustainable solution, let $\lambda_j(0)=\lambda_j(T)$, the unique $\lambda_j(t)=q_j \cdot L(t)$, in the form of hybrid limit cycle, and equilibrium initial value  $\lambda_j(0)=\lambda_{j_{hlc}}=q_j(m_1-\frac{1}{s_1})$ can also be derived.

Therefore, the Nash equilibrium of player $j$ is obtained as follows:

\begin{equation}\label{eq:$uNE$}
v^{NE}_{j}=\begin{cases}b_j,&\lambda_j(t)>0,\\ b_j+\frac{\xi_j}{a_j}\lambda_j(t),&\lambda_j(t) \in[-\frac{a_jb_j}{\xi_j},0],\\ 0,& \lambda_j(t)<-\frac{a_jb_j}{\xi_j}.\end{cases}
\end{equation}


\subsection{Partial-cooperative scenario}
\subsubsection{Two farsighted players cooperate: $\pi_3=\{\{1,2\},\{3\}\}$}

When myopic player acts as a singleton and two farsighted players cooperate. For player $3$, we have
\begin{equation}\label{eq:$u3pi_3$}
v_3^{\pi_3}=b_3.
\end{equation}

In coalition $\{1,2\}$, players maximize their joint payoff
\begin{equation}\label{eq:$kpa3$}
K_{\{1,2\}}=K_1+K_2=\sum_{i=1}^{2}\int_{0}^{\infty} e^{-\rho t} [\sum_{i=1}^{2}(a_iv_i(b_i-v_i/2)) - qz]\, dt, \\
\end{equation}
where $q=q_1+q_2$. 

Employing the Pontryagin's Maximum Principle, the current value Hamiltonian of the player $j=1,2$ is given by
\begin{equation}\label{eq:$Hpi3$}
 H_{\{1,2\}}(z,u,\lambda)=\sum_{i=1}^{2}(a_iv_i(b_i-v_i/2)) - qz+ \lambda_{\{1,2\}}(\sum_{i=1}^{3}\xi_i v_i-\delta(t) z).
\end{equation}

It is straightforward that $\lambda_{\{1,2\}}=\lambda$.


Hence, the optimal control of player $j$ is obtained as follows:
\begin{equation}\label{eq:$upi3$}
v^{\pi_3}_{j}=\begin{cases}b_j,&\lambda(t)>0,\\ b_j+\frac{\xi_j}{a_j}\lambda(t),&\lambda(t) \in[-\frac{a_jb_j}{\xi_j},0],\\ 0,& \lambda(t)<-\frac{a_jb_j}{\xi_j}.\end{cases}
\end{equation}

\subsubsection{Farsighted player $1$ cooperates with myopic player: $\pi_{41}=\{\{1,3\},\{2\} \}$}

In this scenario, for coalition $\{1,3\}$, players maximize their joint payoff
\begin{equation}\label{eq:$kpa41$}
K_{\{1,3\}}=K_1+K_3=\int_{0}^{\infty} e^{-\rho t} [a_1v_1(b_1-v_1/2)+a_3v_3(b_3-v_3/2) - q_1z]\, dt. \\
\end{equation}

Then, the current value Hamiltonian of the player $j=1,3$ is given by
\begin{equation}\label{eq:$Hpi41$}
 H_{\{1,3\}}(z,u,\lambda)=a_1v_1(b_1-v_1/2)+a_3v_3(b_3-v_3/2) - q_1z+ \lambda_{\{1,3\}}(\sum_{i=1}^{3}\xi_i v_i-\delta(t) z).
\end{equation}

The dynamics of the adjoint variable is derived by 
\begin{equation}\label{eq:$lampi41$}
\dot{\lambda}_{\{1,3\}}=\rho\lambda_{\{1,3\}}-\frac{\partial  H_{\{1,3\}}}{\partial z}=q_1+(\rho+\delta(t))\lambda_{\{1,3\}}.
\end{equation}
Hence, the optimal control of player $j=1,3$ is obtained as follows:
\begin{equation}\label{eq:$upi41$}
v^{\pi_{41}}_{j}=\begin{cases}b_j,&\lambda_1(t)>0,\\ b_j+\frac{\xi_j}{a_j}\lambda_1(t),&\lambda_1(t) \in[-\frac{a_jb_j}{\xi_j},0],\\ 0,& \lambda_1(t)<-\frac{a_jb_j}{\xi_j}.\end{cases}
\end{equation}

The player $2$ acts as a singleton in this coalition structure. Therefore, his control strategy, equivalent to his Nash equilibrium in the noncooperative scenario $\pi_2$, is given by
\begin{equation}\label{eq:$upi412$}
v^{\pi_{41}}_{2}=\begin{cases}b_2,&\lambda_2(t)>0,\\ b_2+\frac{\xi_2}{a_2}\lambda_2(t),&\lambda_2(t) \in[-\frac{a_2b_2}{\xi_2},0],\\ 0,& \lambda_2(t)<-\frac{a_2b_2}{\xi_2}.\end{cases}
\end{equation}

\subsubsection{Farsighted player $2$ cooperate with myopic player: 
 $\pi_{42}=\{ \{2,3\}, \{1\} \}$}

As the coalition $\pi_{42}$ is similar to $\pi_{41}$, we skip the calculation process and present the result directly.

Optimal control of player $j=2,3$ is given by

\begin{equation}\label{eq:$upi42$}
v^{\pi_{42}}_{j}=\begin{cases}b_j,&\lambda_2(t)>0,\\ b_j+\frac{\xi_j}{a_j}\lambda_2(t),&\lambda_2(t) \in[-\frac{a_jb_j}{\xi_j},0],\\ 0,& \lambda_2(t)<-\frac{a_jb_j}{\xi_j}.\end{cases}
\end{equation}

Strategy of player $1$ is equivalent to his Nash equilibrium in the noncooperative scenario $\pi_2$.
\begin{equation}\label{eq:$upi422$}
v^{\pi_{42}}_{1}=\begin{cases}b_1,&\lambda_1(t)>0,\\ b_1+\frac{\xi_1}{a_1}\lambda_1(t),&\lambda_1(t) \in[-\frac{a_1b_1}{\xi_1},0],\\ 0,& \lambda_1(t)<-\frac{a_1b_1}{\xi_1}.\end{cases}
\end{equation}

To sum up, according to the characteristic of players’ strategy under different coalition structures, we propose the following remark:
\begin{remark}\label{rem01}
  When the myopic player cooperates with any number of the farsighted players, following the farsighted player(s) in the same coalition, the strategy of the myopic player is transformed into the form of a hybrid limit cycle, instead  of taking the maximum production rate, as a result of the interaction between players with different perspectives.
\end{remark}

\begin{remark}\label{rem02}
  If the initial stock is fixed to $z_0$, the cooperation coalition, denoted by $\pi_1$, is observed to maintain the stock of pollution at its lowest level compared to other feasible coalition structures
\end{remark}

From the insight shown above, in order to keep the stock of environmental pollution at a minimum level and guarantee players' best possible integral payoff, players have to cooperate in the grand coalition $\pi_1$.  This may occur for the reason that players' emission rate, as well as the stock of pollution govern by the players' emission strategy, under $\pi_1$ are lower than that in other coalition structures.

\section{Dynamics of the state under different coalition structures}

In the previous section, we specified players' strategies under coalition structure $\pi=\pi_1,\ \pi_2,\ \pi_3,\ \pi_{41}, \ \pi_{42}$. 
Subsequently, we turn to examining the dynamics of the state under these scenarios. The proof of the forthcoming theorem relies on \citep{Gromov2023}.

 \begin{theorem}\label{Theorem: Theorem1}
 
Under coalition structure $\pi =\pi_1,\ \pi_2,\ \pi_3,\ \pi_{41}, \ \pi_{42}$, for an arbitrary initial state $z^{\pi}(0)=z^{\pi}_0\ge 0$, the state variable $z^{\pi}(t)$ driven by $\{v_{i}^{\pi}\}_{1}^{3}$ exponentially converges to a unique hybrid limit cycle $\bar{z}^{\pi}(t)$ as $t\to \infty$.
\end{theorem}

\begin{proof}
The dynamics of the state $z^{\pi}(t) \in z$ is defined in the following way:

\begin{equation}\label{eq:$xpi$}
\dot{z}^{\pi}(t)=\sum_{i=1}^{3}\xi_i v_{i}^{\pi}(t)-\delta(t) z^{\pi}(t), \ z^{\pi}(0)=z^{\pi}_0, 
\end{equation}
where  $v_{i}^{\pi}$ is uniquely determined in the previous section.

The general solution of \eqref{eq:$xpi$} can be derived as follows:

\begin{equation}\label{eq:$xsol$}
\phi(t,z^{\pi}_0)=z^{\pi}_0e^{-\int_{0}^{t}\delta(s) ds}+e^{-\int_{0}^{t}\delta(s) ds} \int_{0}^{t}\sum_{i=1}^{3}\xi_i v_{i}^{\pi}(\tau)e^{\int_{0}^{\tau}\delta(s) ds}d\tau.
\end{equation}

As $\phi(T,z^{\pi}_0)$ is the mapping from a compact convex set $z$ to $z$. By utilizing the Brouwer's fixed-point theorem, there exists a point $\bar{z}^{\pi}_{hlc}$,  where $\phi(T,\bar{z}^{\pi}_{hlc})=\bar{z}^{\pi}_{hlc}$.  Furthermore, the environmentally sustainable $v_{i}^{\pi}$ and self-purify rate $\delta(t)$ are periodic functions. Thus, according to the general solution \eqref{eq:$xsol$}, it follows that $\phi(t,\bar{z}^{\pi}_{hlc})=\phi(t+kT,\bar{z}^{\pi}_{hlc}),\ k\in \mathbb{N}_0$. 

Therefore, there exists a periodic solution  $\bar{z}^{\pi}(t)$ of (\ref{eq:$xpi$}) with period $T$  and initial condition $\bar{z}^{\pi}(0)=\bar{z}^{\pi}_{hlc}$. Moreover, by observing the steady state $\bar{z}^{*\pi}(t)=\frac{\sum_{i=1}^{3}\xi_i v_{i}^{\pi}(t)}{\delta(t)}$, as $v_{i}^{\pi}(t)$ is a hybrid limit cycle, in each subperiod $\bar{z}^{*\pi}(t)$ varies monotonically. Hence, the periodic state $\bar{z}^{\pi}(t)$, attracted by $\bar{z}^{*\pi}(t)$, varies monotonicially in each subperiod.  It implies that $\bar{z}^{\pi}(t)$ is not only a periodic solution but also a unique hybrid limit cycle starting from $\bar{z}^{\pi}_{hlc}$.

\textbf{Uniqueness:}

In addition, we proceed to prove the uniqueness of such a solution.

Noting that under different coalition structures, farsighted players' strategies are determined in the form of hybrid limit cycle, and the strategy of myopic player is either hybrid limit cycle or constant. Thus, their strategies are all periodic solutions, which implies that the second term in the right-hand side of \eqref{eq:$xsol$} does not depend on the initial value of the state. Thus, for any other initial value $z^{\pi\prime}_{hlc}$ such that $\phi(T,z^{\pi\prime}_{hlc})=z^{\pi\prime}_{hlc}$, we have

$0 \leq |z^{\pi\prime}_{hlc}-\bar{z}^{\pi}_{hlc}|= |\phi(T,z^{\pi\prime}_{hlc})- \phi(T,\bar{z}^{\pi}_{hlc})| = |z^{\pi\prime}_{hlc}-\bar{z}^{\pi}_{hlc}|e^{-\int_{0}^{T}\delta(s) ds} \leq |z^{\pi\prime}_{hlc}-\bar{z}^{\pi}_{hlc}|e^{-\delta_{min}T}$.

Here, $\delta_{min}=\min{(\delta_1,\delta_2)}$.
Then, it follows that

$0 \leq |z^{\pi\prime}_{hlc}-\bar{z}^{\pi}_{hlc}|(1-e^{-\delta_{min}T})\leq 0$.

Therefore, we have $z^{\pi\prime}_{hlc}=\bar{z}^{\pi}_{hlc}$.

\textbf{Convergence:}

Obviously, since $v_{i}^{\pi}(\tau)>0$ and $\delta(\tau)>0$ is independent of the initial value $z^{\pi}(0)$, for any arbitrary initial values $z^{\pi}(0)$,  we obtain the following exponential convergence property:

 $\lim\limits_{t \to \infty}| \phi(t,z^{\pi}(0)) - \phi(t,\bar{z}^{\pi}_{hlc}) | = \lim\limits_{t \to \infty} |z^{\pi}(0)- \bar{z}^{\pi}_{hlc}|e^{-\int_{0}^{t}\delta(s) ds}=0.$

This completes the proof.
\end{proof}

As the Theorem \ref{Theorem: Theorem1} mentioned above, we find that for any possible coalition structure, the dynamics of the state converges to the uniquely defined hybrid limit cycle $\bar{z}^{\pi}(t)$. Furthermore, among them, the cooperation structure $\pi_1$ brings the lowest convergent pollution stock as the players strategies are more conservative.

\section{Stability of cooperation}

\subsection{Sustainably-cooperative optimality principle}

 Cooperative behavior is optimal in terms of maximizing total utility. However, deviating from cooperative behavior may be beneficial to one or a group of participants. This raises the question of the stability of the cooperative agreement. In the context of stationary problems, the stability of cooperative partitions is the subject of works such as \citep{Sedakov13, Sun21, su2023}. According to the idea introduced in these works, a cooperative scenario is stable with respect to individual deviations (Nash stable coalition structure) if none of the players benefits from leaving the maximal coalition and becoming a singleton.
Thus, in the considered three-player game, the deviation of player 3 from the maximum coalition leads to a transition from coalition structure $\pi_1$ to coalition structure $\pi_3$, and if players 1 or 2 deviate --- to structures $\pi_{42}$ and $\pi_{41}$, 
 respectively. If this turns out to be to the advantage of the deviating player, then the coalition  structure $\pi_1$ will not be stable. Based on this idea, we propose a cooperative optimality principle for the problem under consideration, which includes imputations that satisfy the above property of stability against individual deviations. Since we are dealing with a dynamical model, we also investigate the non-emptiness of this optimality principle for each subgame along the cooperative trajectory. The time-consistency \citep{Pet2021, Pet2003, zacc2008} of the introduced optimality principle is also investigated.
 To keep the players in the grand coalition throughout the game, the method of designing the imputation distribution mechanism is applied. The condition for the stability of the grand coalition with two developed countries and one developing is examined in \citep{su2023}, however, the time-consistency issue was not considered.




By $K_i^{\pi}(z^{\pi_1}(\epsilon),v)$ we denote the payoff functional of player $i$ under coalition structure $\pi$ in a subgame starting from the  moment $\epsilon$ with initial state $z^{\pi_1}(\epsilon)$
\begin{equation}\label{eq:$subgameK$}
K_i^{\pi}(z^{\pi_1}(\epsilon),v)=\int_{\epsilon}^{\infty} e^{-\rho (t-\epsilon)} [\sum_{i=1}^{3}(a_iv_i(b_i-v_i/2)) - q_iz]\, dt,\\
\end{equation}
where $q_1,q_2>0$ and $q_3=0$.

Denote  by  $v(S, z^{\pi_1}(t),t)$  the characteristic function of coalition $S\subseteq N$ in the subgame $\Gamma(z^{\pi_1}(t),t)$ starting from $t\in [0,\infty]$ with initial state $z^{\pi_1}(t)$, which maps coalition $S$ to a numerical value and can be interpreted as a measure of coalition's power (payoff, strength). We consider the classic characteristic functions as shown in \citep{von1947}.

Define by $R(z^{\pi_1}(t),t)$  the imputation set of the 3-player game in the subgame $\Gamma(z^{\pi_1}(t),t)$ starting from $t\in [0,\infty]$ with initial state $z^{\pi_1}(t)$:
\begin{equation}
    \begin{aligned}
    R(z^{\pi_1}(t),t)=\bigg \{\eta(z^{\pi_1}(t),t)=(\eta_1(z^{\pi_1}(t),t), \eta_2(z^{\pi_1}(t),t),\eta_3(z^{\pi_1}(t),t)):\\
     \sum_{i=1}^{3} \eta_i(z^{\pi_1}(t),t)= v(N,z^{\pi_1}(t),t), \ \eta_i(z^{\pi_1}(t),t) \geq   v(\{i\},z^{\pi_1}(t),t),\ i\in N \bigg \}.
   \end{aligned}
\end{equation}


Imputation is a vector that satisfies  the first condition of group rationality  and the second condition of individual 
 rationality above. The concept of group rationality means that the collective benefits of cooperation within the grand coalition are fully distributed among players. Individual rationality refers to the idea that payoffs where a player obtains less than it could receive by acting alone, as measured by player's characteristic function value $v(\{i\},z^{\pi_1}(t),t)$, are unacceptable. 

 \begin{myDef}\label{def3}
In the 3-player subgame $\Gamma(z^{\pi_1}(t),t)$ with heterogeneous players the \textbf{sustainably-cooperative optimality principle}  $Z(z^{\pi_1}(t),t)$  is the set of vectors, such that
\begin{equation}\label{suq_coop}
    \begin{aligned}
    Z(z^{\pi_1}(t),t)=\bigg \{\zeta(z^{\pi_1}(t),t)=(\zeta_1(z^{\pi_1}(t),t), \zeta_2(z^{\pi_1}(t),t),\zeta_3(z^{\pi_1}(t),t)):\\
     \sum_{i=1}^{3} \zeta_i(z^{\pi_1}(t),t)=v(N,z^{\pi_1}(t),t), \ \zeta_1(z^{\pi_1}(t),t) \geq K_1^{\pi_{42}}(z^{\pi_1}(t),v^{\pi_{42}}(t)),
\\ \zeta_2(z^{\pi_1}(t),t) \geq K_2^{\pi_{41}}(z^{\pi_1}(t),v^{\pi_{41}}(t)), 
 \zeta_3(z^{\pi_1}(t),t) \geq K_3^{\pi_{3}}(z^{\pi_1}(t),v^{\pi_{3}}(t) \bigg \}. 
   \end{aligned}
\end{equation}
where $\zeta_i(z^{\pi_1}(t),t)$, for $i=1,2,3$, is the payment that allocated to the player $i$ from the grand coalition in the subgame $\Gamma(z^{\pi_1}(t),t)$ and $v(N,z^{\pi_1}(t),t)= \sum_{i=1}^{3}K_i^{\pi_1}(z^{\pi_1}(t),v^{\pi_1}(t))$.
\end{myDef}

\begin{lemma}\label{le01}
The sustainably-cooperative optimality principle $Z(z^{\pi_1}(t),t)$ is a subset of the imputation set $R(z^{\pi_1}(t),t)$ in any subgame $\Gamma(z^{\pi_1}(t),t)$, starting from point $t\in[0,\infty]$, along the cooperative trajectory $z^{\pi_1}(t)$, i.e., $Z(z^{\pi_1}(t),t) \subset R(z^{\pi_1}(t),t)$.
\end{lemma}

\begin{proof}
    
    The definition of the characteristic function in \citep{von1947} is as follows

\begin{equation}
 v(S, z(t),t)=\begin{cases}0,&S=\{\emptyset\},\\ 
\mathop{\max}\limits_{v_i, \atop i\in S}  \mathop{\min}\limits_{v_j,\atop j\in N\setminus  S}   \sum\limits_{i\in S} K_i(z(t),v), &S\subset N,\\ 
\mathop{\max}\limits_{v} \sum\limits_{i=1}^{n} K_i(z(t),v),& S=N.
\end{cases}
\end{equation}

Then, $v(\{i\},z^{\pi_1}(t),t)=\mathop{\max}\limits_{v_i}  \mathop{\min}\limits_{v_j,\atop j\in N\setminus i} K_i(z^{\pi_1}(t),v(t))$ in which player $i$ maximizes his/her payoff and other players in $N \setminus i$ minimize  it. Note, that
 $v(\{1\},z^{\pi_1}(t),t)=K_1(z^{\pi_1}(t),v_1^{NE}(t), b_2, b_3 )$,  $v(\{2\},z^{\pi_1}(t),t)=$\linebreak $ K_2(z^{\pi_1}(t),b_1, v_2^{NE}(t), b_3 )$, $v(\{3\},z^{\pi_1}(t),t)=K_3(z^{\pi_1}(t),b_1, b_2, b_3)$. 
 For farsighted  player $i=1,2$, player $i$ adopts Nash equilibrium strategy $v^{NE}_i$, while other players ramp up their production rate to the maximum level in order to increase the stock of pollution as possible. This results in increasing of the tax ($q_i z$) that player $i$ has to bear. For the myopic player, whose payoff only depends on his/her own strategy,  he/she adopts Nash equilibrium strategy ($v^{NE}_3=b_3$). 

We find that the strategies of the singleton player in the partial-cooperative case and in the above situation are the same, however, the stock of pollution is greater in the above case. Thus, we have
$K_1^{\pi_{42}}(z^{\pi_1}(t),v^{\pi_{42}}(t)) > v(\{1\},z^{\pi_1}(t),t), K_2^{\pi_{41}}(z^{\pi_1}(t), v^{\pi_{41}}(t))>v(\{2\},z^{\pi_1}(t),t), K_3^{\pi_{3}}(z^{\pi_1}(t),v^{\pi_{3}}(t))=v(\{3\},z^{\pi_1}(t),t)$.

Therefore, we conclude that $Z(z^{\pi_1}(t),t)$ is the subset of $R(z^{\pi_1}(t),t)$.

\end{proof}



The reason behind this sustainably-cooperative optimality principle is straightforward. For the subgame starting at moment $t\in[0,\infty]$,
along  the trajectory $z^{\pi_1}(t)$, we should always guarantee that the distributed payment to the player $i$ in the grand coalition is greater than it would be if player $i$ deviated from the grand coalition and played as a singleton.

For $j=1,2$, we define $f_j(t)=z^{\pi_{4j}}(t)-z^{\pi_1}(t), \ \overline{f_j}(t)=\bar{z}^{\pi_{4j}}(t)-\bar{z}^{\pi_1}(t)$, $M(t)=q_1f_2(t)+q_2f_1(t)$ and  $\overline{M}(t)=q_1\overline{f_2}(t)+q_2\overline{f_1}(t)$.
 The following expression can be obtained for $\overline{M}(t)$:
\begin{equation}\label{eq:$over_M$}
\overline{M}(t)=\begin{cases}A_1+B_1e^{-\delta_1 (t-kT)}+C_1e^{s_1 (t-kT)},& t\in[kT,(k+ \tau)T),\\A_2+B_2e^{-\delta_2 (t-kT)}+C_2e^{s_2 (t-kT)},& t\in[(k+\tau)T),(k+1)T)],\end{cases}
\end{equation}
where
\begin{equation*}
\begin{aligned}
& \{k_i\}_{1}^{3}=\{\frac{\xi_i^2}{2a_i}\}_{1}^{3},\ o_m=-2[(k_1+k_2)q_1q_2+(k_1+k_3)q_1^2+(k_1+k_3)q_2^2)],\\
& a_{m1}=o_m m_1,\ a_{m2}=o_m m_2,\ b_{m1}=\frac{o_m}{s_1},\ b_{m2}=\frac{o_m}{s_2}, \ \overline{M}_0=\overline{M}(0),\\ 
& A_1=-\frac{b_1}{\delta_1}, \ B_1=\frac{\delta_1(b_{m1}-a_{m1}+s_1\overline{M}_0)+b_{m1}s_1+\overline{M}_0\delta_1^2}{\delta_1(s_1+\delta_1)},\ C_1=\frac{a_{m1}}{s_1+\delta_1}, \\ 
& A_2=-\frac{b_2}{\delta_2},\ B_1=e^{\delta_2 T}\frac{\delta_2(b_{m2}-a_{m2}e^{s_2T}+s_2\overline{M}_0)+b_{m2}s_2+\overline{M}_0\delta_2^2}{\delta_2(s_2+\delta_2)},\ C_2=\frac{a_{m2}}{s_2+\delta_2}.
\end{aligned}
\end{equation*}

Then, we propose the  proposition to fulfil the above defined condition in Definition \ref{def3} for all $\epsilon \in [0,\infty]$.

\begin{proposition}\label{Proposition1}
 The sustainably-cooperative optimality principle  is not empty along the cooperative trajectory, i.e.  $Z(z^{\pi_1}(\epsilon),\epsilon) \neq \emptyset$, $\forall \epsilon\in [0,\infty]$,  if and only if the following condition is satisfied:
\begin{equation}\label{Prop1}
    Y \leq 
    \begin{large}
         \begin{cases} \frac{I_{\tau T}+ E}{G_0},& if \ s_1>s_2,\\ \frac{I_{0} + E}{G_0},& if \ s_1<s_2,\end{cases}
    \end{large}
\end{equation}
where 
\begin{equation*}
\begin{aligned}
&Y=k_3q^2+k_1(q^2-q_1^2)+k_2(q^2-q_2^2), \\
&I_{\tau T}=\frac{(A_1+B_1e^{-\delta_1\tau T}+C_1e^{s_1\tau T})e^{(2s_1\tau +s_2(1-\tau))T}}{e^{(s_1\tau +s_2(1-\tau))T}-1}[\frac{1-e^{-s_1\tau T}}{s_1}\\
&-\frac{e^{(s_2-s_1)\tau T}(e^{-s_2T}-e^{-s_2\tau T})}{s_2}]-\frac{e^{s_1\tau T}-1}{s_1},\\
&  I_{0}= \frac{\overline{M}_0 e^{(s_1\tau +s_2(1-\tau))T}}{e^{(s_1\tau +s_2(1-\tau))T}-1}[\frac{1-e^{-s_1\tau T}}{s_1}
-\frac{e^{(s_2-s_1)\tau T}(e^{-s_2T}-e^{-s_2\tau T})}{s_2}],\\
&  E=\frac{e^{\rho T}}{e^{\rho T}-1}[\frac{C_1(-1+e^{\delta_1 \tau T})}{\delta_1}-\frac{C_2(e^{\delta_2T}+e^{\delta_2 \tau T})}{\delta_2}\\
&+\frac{A_1(1-e^{-\rho \tau T})+A_2(e^{-\rho \tau T}-e^{-\rho T})}{\rho}+\frac{B_1(1-e^{-s_1 \tau T})}{s_1}+\frac{B_2(e^{-s_2\tau T-e^{-s_2 T}})}{s_2}],\\
& G_0=\frac{e^{\rho T}}{e^{\rho T}-1}[\frac{-2m_1(-1+e^{\delta_1 \tau T})}{s_1 \delta_1} +\frac{m_1^2(-1+e^{(2s_1-\rho)\tau T})}{2s_1-\rho}\\
&+\frac{2m_2(e^{\delta_2 \tau T}-e^{\delta_2 T})}{s_2 \delta_2}+\frac{m_2^2 (e^{(2s_2-\rho)T}-e^{(2s_2-\rho)\tau T})}{2s_2-\rho}+\frac{e^{-\rho \tau T}(s_1^2-s_2^2)+s_2^2-s_1^2e^{-\rho T}}{s_1^2s_2^2\rho}].
\end{aligned}
\end{equation*}

\end{proposition}

\begin{proof}
 According to the Definition \ref{def3} optimality principle  $Z(z^{\pi_1}(\epsilon),\epsilon)$ is not empty along the cooperative trajectory if $\forall \epsilon\in [0,\infty]$ the following inequality holds:
\begin{equation}\label{inequality}K^{\pi_1}(z^{\pi_1}(\epsilon),v^{\pi_1})- K_1^{\pi_{42}}(z^{\pi_1}(\epsilon),v^{\pi_{42}}) -K_2^{\pi_{41}}(z^{\pi_1}(\epsilon),v^{\pi_{41}})- K_3^{\pi_{3}}(z^{\pi_1}(\epsilon),v^{\pi_{3}})\geq 0.\end{equation} 

Note that
\begin{equation*}
\begin{aligned}
&K^{\pi_1}(z^{\pi_1}(\epsilon),v^{\pi_1})- K_1^{\pi_{42}}(z^{\pi_1}(\epsilon),v^{\pi_{42}}) -K_2^{\pi_{41}}(z^{\pi_1}(\epsilon),v^{\pi_{41}})- K_3^{\pi_{3}}(z^{\pi_1}(\epsilon),v^{\pi_{3}}) \\
&=\int_{\epsilon}^{\infty} e^{-\rho (t-\epsilon)}L^2(t)[k_1(q^2_1-q^2)+k_2(q^2_2-q^2)-k_3q^2]\ dt + \int_{\epsilon}^{\infty} e^{-\rho (t-\epsilon)} M(t)\ dt.
\end{aligned}
\end{equation*}

Let $G(\epsilon)=\int_{\epsilon}^{\infty} e^{-\rho t} L^2(t) dt,\ Y=k_3q^2+k_1(q^2-q_1^2)+k_2(q^2-q_2^2)$, the inequality (\ref{inequality})  can be transformed into the following form:
\begin{equation}\label{eq:minepsi}
    Y \leq \mathop{\min}_{\epsilon}[\frac{\int_{\epsilon}^{\infty} e^{-\rho t} M(t)\ dt}{G(\epsilon)}].
\end{equation}

Recall that $L(t)$ is an uniquely defined hybrid limit cycle, for $s_1>s_2$ we have $m_1<0$ and $m_2>0$, which indicates that $L(t)$ in every first subperiod decreases, then increases in the second subperiod. And vice versa for $s_1<s_2$. 

As $L^2(t)$ and $\bar{M}(t)$ are periodic functions, we have $G(\epsilon)=e^{-\rho \epsilon}G(0)$  and  $ \int_{\epsilon}^{\infty} e^{-\rho t} \bar{M}(t)\ dt= e^{-\rho \epsilon}\int_{0}^{\infty} e^{-\rho t} \bar{M}(t)\ dt$. 

For $\pi=\pi_1,\pi_2,\pi_3,\pi_{41},\pi_{42}$, in the subgame $\Gamma(z^{\pi_1}(\epsilon),\epsilon),\ t\geq\epsilon$,  along the cooperative trajectory $z^{\pi_1}(\epsilon)$, we have  
\begin{equation}\label{eq:$subgamex$}
z^{\pi}(t)= \phi(t,z^{\pi_1}(\epsilon))=e^{- \int_{\epsilon}^{t}\delta(s)ds}[z^{\pi_1}(\epsilon) + \int_{\epsilon}^{t}\sum_{i=1}^{3}\xi_i v_{i}^{\pi}(\tau)e^{\int_{\epsilon}^{\tau}\delta(s) ds}d\tau].
\end{equation}

Then, with $z^{\pi_{4j}}(\epsilon)=z^{\pi_1}(\epsilon),\ j=1,2$,
\begin{equation}\label{simplify01}
\begin{aligned}
f_j(t)=z^{\pi_{4j}}(t) &-z^{\pi_{1}}(t)=z^{\pi_{4j}}(t) - \bar{z}^{\pi_{4j}}(t)- (z^{\pi_{1}}(t) - \bar{z}^{\pi_{1}}(t)) + \bar{z}^{\pi_{4j}}(t) -\bar{z}^{\pi_{1}}(t) \\
&=e^{- \int_{\epsilon}^{t}\delta(s)ds} (\bar{z}^{\pi_{1}}(\epsilon)-\bar{z}^{\pi_{4j}}(\epsilon))+ \bar{z}^{\pi_{4j}}(t) -\bar{z}^{\pi_{1}}(t)\\
&=-\bar{f}_j(\epsilon) e^{- \int_{\epsilon}^{t}\delta(s)ds} + \bar{f}_j(t)
\end{aligned}
\end{equation}

Hence, in the subgame $\Gamma(z^{\pi_1}(\epsilon),\epsilon),\ t\geq\epsilon$, through the derived general solution of $z^{\pi}$ in \eqref{eq:$subgamex$} and \eqref{simplify01}, we simplify

\begin{equation*}
\begin{aligned}
&\frac{\int_{\epsilon}^{\infty} e^{-\rho t} M(t)\ dt}{G(\epsilon)} =\frac{-\overline{M}(\epsilon) \int_{\epsilon}^{\infty}e^{-\rho t- \int_{\epsilon}^{t}\delta(s)ds} \ dt +\int_{\epsilon}^{\infty}e^{-\rho t}\overline{M}(t) \ dt}{G(\epsilon)} \\
&=\frac{-\overline{M}(\epsilon) \int_{\epsilon}^{\infty}e^{-\rho t- \int_{\epsilon}^{t}\delta(s)ds} \ dt + e^{-\rho \epsilon}\int_{0}^{\infty}e^{-\rho t}\overline{M}(t) \ dt}{e^{-\rho \epsilon}G(0)} \\
& = \frac{-\overline{M}(\epsilon) \int_{\epsilon}^{\infty}e^{-\rho (t-\epsilon)- \int_{\epsilon}^{t}\delta(s)ds} \ dt +\int_{0}^{\infty}e^{-\rho t}\overline{M}(t) \ dt}{G(0)}.
\end{aligned}
\end{equation*}

Let $h(\epsilon)=\int_{\epsilon}^{\infty}e^{-\rho (t-\epsilon)- \int_{\epsilon}^{t}\delta(s)ds} \ dt$. Thus,
\begin{equation*}
\begin{aligned}
\mathop{\arg\min}_{\epsilon} [\frac{-\overline{M}(\epsilon) h(\epsilon) +\int_{0}^{\infty}e^{-\rho t}\overline{M}(t) \ dt}{G(0)}]= \mathop{\arg\max}_{\epsilon} [\overline{M}(\epsilon) h(\epsilon)].
\end{aligned}
\end{equation*}

For $ o_m=-2[(k_1+k_2)q_1q_2+(k_1+k_3)q_1^2+(k_1+k_3)q_2^2)]$, the dynamics of $\overline{M}(\epsilon)$ is given by

\begin{equation*}
\begin{aligned}
\dot{\overline{M}}(\epsilon)&= q_1\dot{\bar{f}}_2(t)+q_2\dot{\bar{f}}_1(t)\\
&=o_mL(\epsilon)-\delta(\epsilon)\overline{M}(\epsilon).
\end{aligned}
\end{equation*}

The steady state of $\overline{M}(\epsilon)$ is derived as follows
\begin{large}
\begin{equation*}
\overline{M}^{*}(\epsilon)= \begin{cases} \frac{o_m L(\epsilon)}{\delta_1},& \epsilon \in [kT, (k+\tau)T],\\ \frac{o_m L(\epsilon)}{\delta_2},& \epsilon \in [(k+\tau)T, (k+1)T].\end{cases}
\end{equation*}
\end{large}


For $s_1>s_2\ (s_1<s_2)$, $L(\epsilon)$ in each period $(T)$ first decreases (increases) and then increases (decreases). Moreover, for $\epsilon \in [0, \infty]$, $\overline{M}(\epsilon)$ is a hybrid limit cycle varying with the period $T$. Therefore, for $s_1>s_2\ (s_1<s_2)$, $\overline{M}(\epsilon)$ in each period $(T)$ first increases (decreases) and then decreases (increases). 

Next, we move to study $h(\epsilon)=\int_{\epsilon}^{\infty}e^{-\rho (t-\epsilon)- \int_{\epsilon}^{t}\delta(s)ds} \ dt$. Let $k_{\epsilon}=\lfloor \frac{\epsilon}{T}\rfloor T,\ o=\frac{1-e^{-s_1\tau T}}{s_1}-\frac{e^{(P_2-P_1)\tau T}(e^{-s_2T}-e^{-s_2\tau T})}{s_2}$. The explicit form of $h(\epsilon)$ is derived as the following expressions:

\begin{equation*}
h(\epsilon)= \begin{cases} h_1(\epsilon),& if \ 0 \leq \epsilon-k_{\epsilon} \leq \tau T,\\ h_2(\epsilon),& if \ \tau T \leq \epsilon-k_{\epsilon} \leq T,\end{cases}
\end{equation*}
where 
\begin{equation*}
\begin{aligned}
h_1(\epsilon)=&\frac{e^{(s_1\tau+s_2(1-\tau))T+s_1(\epsilon-k_{\epsilon})}}{e^{(s_1\tau+s_2(1-\tau))T}-1} o- \frac{e^{s_1(\epsilon-k_{\epsilon})}-1}{s_1} ,\\
h_2(\epsilon)=&\frac{e^{(s_1\tau+s_2(1-\tau))T+s_1\tau T+(\epsilon-k_{\epsilon}-\tau T)s_2}}{e^{(s_1\tau+s_2(1-\tau))T}-1} o\\ &-\frac{e^{s_1\tau T+(\epsilon-k_{\epsilon}-\tau T)s_2}-e^{(\epsilon-k_{\epsilon}-\tau T)s_2}}{s_1}+\frac{1-e^{(\epsilon-k_{\epsilon}-\tau T)s_2}}{s_2}.
\end{aligned}
\end{equation*}

The first order derivative of $h(\epsilon)$ is

\begin{equation*}
\frac{\mathrm{d} h(\epsilon) }{\mathrm{d} \epsilon} = 
\begin{large}
\begin{cases} e^{s_1(\epsilon-k_{\epsilon})} [\frac{(s_2-s_1)(1-e^{s_2(1-\tau)T})}{s_2(e^{(s_1\tau+s_2(1-\tau))T}-1)}],& if \ 0 \leq \epsilon-k_{\epsilon} \leq \tau T, \\
e^{s_2(\epsilon-k_{\epsilon}-\tau T)} [\frac{(s_2-s_1)(e^{s_1\tau T}-1)}{s_1(e^{(s_1\tau+s_2(1-\tau))T}-1)}],& if \ \tau T \leq \epsilon-k_{\epsilon} \leq T.
\end{cases}
\end{large}
\end{equation*}

It is very easy to check that when $s_1>s_2 \ (s_1<s_2)$, in the first subperiod of every period $T$, $\frac{\mathrm{d} h(\epsilon) }{\mathrm{d} \epsilon}>0 \ (\frac{\mathrm{d} h(\epsilon) }{\mathrm{d} \epsilon}<0)$, while in the second subperiod, $\frac{\mathrm{d} h(\epsilon) }{\mathrm{d} \epsilon}<0 \ (\frac{\mathrm{d} h(\epsilon) }{\mathrm{d} \epsilon}>0)$. As $h(\epsilon)=h(\epsilon+kT)$, meanwhile, $h(\epsilon)>0$ has the same trend as $\overline{M}(\epsilon)>0$. We conclude that function $\overline{M}(\epsilon)h(\epsilon)$ is a hybrid limit cycle. Therefore, 

\begin{equation}\label{eq:argmax}
\mathop{\arg\max}_{\epsilon} [\overline{M}(\epsilon) h(\epsilon)] = 
\begin{cases} \tau T+kT,& if \ s_1>s_2, \\
kT,& if \ s_1<s_2.
\end{cases}
\end{equation}

Then, let $k=0$ and plugging \eqref{eq:argmax} into \eqref{eq:minepsi}. To simplify the notation, define the following equations
\begin{equation}
\begin{aligned}
    I_{\tau T}&=-\overline{M}(\tau T)\int_{\tau T}^{\infty} e^{-\rho (t-\tau T)-\int_{\tau T}^{t}\delta(s)ds}\, dt, \ I_{0}=-\overline{M}(0)\int_{0}^{\infty} e^{-\rho t-\int_{0}^{t}\delta(s)ds}\, dt,\\ E&=\int_{0}^{\infty}e^{-\rho t}\overline{M}(t)\, dt,\ G_0=\int_{0}^{\infty} e^{-\rho t} L^2(t) dt.
\end{aligned}
\end{equation}

Calculating the above equations explicitly, we obtain the proposition.

\end{proof}

Consequently, if and only if there exists a set of parameters satisfying Proposition \ref{Proposition1}, the non-emptiness of the sustainably-cooperative optimality principle $Z(z^{\pi_1}(\epsilon),\epsilon)$ is guaranteed. The payoff of the grand coalition is not less than the sum of the payoffs that the player can obtain when the player deviates from the grand coalition for any subgame starting from $\epsilon \in [0,\infty]$. Through  reasonably distributing payoff of the grand coalition, the stability of cooperation can be achieved as the deviation from the grand coalition would lose player's profit.

\subsection{Time-consistency of the sustainably-cooperative optimality principle}

 Since the optimality principle $Z(z^{\pi_1}(\epsilon),\epsilon)$ is not empty at any instant of time if the above conditions are satisfied, we can further consider an important property of the cooperative solution as time-consistency. In this property,  the parties to the agreement following a cooperative trajectory  adhere to the same optimality principle at all times. As a result, there is no immediate incentive for them to abandon their earlier decision to cooperate.

The notion of time-consistency of the cooperative solution was first introduced by L.A. Petrosyan in 1977 \citep{Petrosyan77}. In \citep{PetrosyanDanilov}, a method was proposed to construct time-consistent cooperative solutions using a special payoff scheme called the imputation distribution procedure (IDP).  In \citep{Pet2003}, the problem of construction of time-consistent IDP was investigated for  cooperative game of pollution reduction with discounted payoffs. Here we detail this notion for the class of games under consideration.

\begin{myDef}\label{defIDP}
A vector function $w(t)=(w_1(t),\ldots, w_n(t))$, $t\in[t_0,T]$
is the imputation distribution procedure (IDP) for imputation
$\eta(z_0,0)\in R(z_0, 0)$, if
$$\eta_i(z_0,0)=\int\limits_{0}^{\infty}e^{-\rho s} w_i(s) \ ds, \ \ i\in N.$$
\end{myDef}

The value of $w_i(t) $  can be interpreted as an instantaneous payout
  to be allocated to player $i$ at instant of time $t$.

\begin{myDef}\label{defTC}
    An optimality principle $P(z_0,0) \subset R(z_0,0)$ is called  time-consistent if
    
    1. $P(z^{\pi_1}(t),t)\neq \emptyset,\ \forall t\in[0,\infty]$
    
    2. For any $\zeta\in P(z_0,0)$, there exists an IDP  $\ w(s)=(w_1(s),...,w_n(s)),\ s\in[0,\infty]$, such that $\zeta_i=\int_{0}^{\infty} e^{-\rho s} w_i(s) \ ds   ,i\in N$ and 
    $$P(z_0,0) \subset \int_{0}^{t} e^{-\rho s} w(s) \ ds  \oplus e^{-\rho t}P(z^{\pi_1}(t),t),\forall t\in [0,\infty].$$
The symbol $\oplus$  indicates $a \oplus B=\{a+b:b\in B\}$, for $a\in R^{n}$ and $B\subset R^{n}$.

\end{myDef}

For optimality principle $Z(z_0,0)$ condition 1 of Definition \ref{defTC} is fulfilled when inequality (\ref{Prop1}) holds. Condition 2  can be interpreted as follows. For each imputation from the optimality principle, there exists such an IDP that, at any instant of time, this imputation can be represented as the sum of the accumulated payoff up to time t (the amount given by the first right-hand term) and the imputation from the same optimality principle in the subgame starting at that time (the second right-hand term). In \citep{Pet2003}, an IDP  that ensures that this condition is fulfilled is called a  time-consistent IDP and a formula for obtaining such a procedure is derived. According to this formula, a time-consistent IDP for imputation $\zeta(z_0)\in Z(z_0,0)$ has the following form:

\begin{equation}\label{tcIDP}
w_i(\epsilon) =\rho \zeta_i(z^{\pi_1}(\epsilon)) -\frac{\mathrm{d}\zeta_i(z^{\pi_1}(\epsilon))}{\mathrm{d} \epsilon}
\end{equation}

To verify this, it is sufficient to check the fulfilment of conditions (\ref{pay_ds1}--\ref{pay_ds}).

\begin{equation}
\label{pay_ds1}
\zeta_i(z^{\pi_1}(0))=\int_{0}^{\infty} e^{-\rho s} w_i(s) \ ds, 
\end{equation}
\begin{equation}
\label{pay_ds}
\int_{0}^{\epsilon} e^{-\rho s} w_i(s) \ ds + e^{-\rho \epsilon}\zeta_i(z^{\pi_1}(\epsilon))=\zeta_i(z^{\pi_1}(0)), \quad \forall  \epsilon \in[0,\infty].
\end{equation}







In order to  get a more computationally convenient form of the time-consistent IDP, we represent the payment to player$i$, $i=1,2,3$, in the subgame $\Gamma(z^{\pi_1}(\epsilon),\epsilon)$ as follows 

\begin{equation}\label{payment}
\begin{aligned}
\zeta_1(z^{\pi_1}(\epsilon))&=K_1^{\pi_{42}}(z^{\pi_1}(\epsilon),v^{\pi_{42}})+\alpha_1 SC(\epsilon),\\
\zeta_2(z^{\pi_1}(\epsilon))&=K_2^{\pi_{41}}(z^{\pi_1}(\epsilon),v^{\pi_{41}}) +\alpha_2 SC(\epsilon),\\
\zeta_3(z^{\pi_1}(\epsilon))&=K_3^{\pi_{3}}(z^{\pi_1}(\epsilon),v^{\pi_{3}}) +\alpha_3 SC(\epsilon),\\
\end{aligned}
\end{equation}
where $SC(\epsilon)$ is the surplus of cooperation $SC(\epsilon)=K^{\pi_{1}}(z^{\pi_1}(\epsilon),v^{\pi_{1}})-K_1^{\pi_{42}}(z^{\pi_1}(\epsilon),v^{\pi_{42}})-K_2^{\pi_{41}}(z^{\pi_1}(\epsilon),v^{\pi_{41}})-K_3^{\pi_{3}}(z^{\pi_1}(\epsilon),v^{\pi_{3}})>0$, and  $\sum\limits^{3}_{i=1}\alpha_i=1,\ \alpha_i\in[0,1]$.
Using the form of imputation (\ref{payment})  and formula (\ref{tcIDP}) we obtain an expression for the time-consistent IDP  in the game under consideration:
\begin{equation}\label{IDP}
\begin{aligned}
 w_i(\epsilon)=&a_iv_i^{NE}(\epsilon)\Big(b_i-\frac{v_i^{NE}(\epsilon)}{2}\Big)-q_iz^{\pi_1}(\epsilon)+\\&\alpha_i\sum\limits_{i=1}^3a_i\Big[v_i^{co}(\epsilon)\Big(b_i-\frac{v_i^{co}(\epsilon)}{2}\Big)-v_i^{NE}(\epsilon)\Big(b_i-\frac{v_i^{NE}(\epsilon)}{2}\Big)\Big].\\
\end{aligned}
\end{equation}

Since the open-loop strategies and the corresponding state variables are periodic over time, it can be easily seen that the time-consistent IDP is also periodic (or seasonal). Thus, by making payments during the game according to the obtained IDP, the players ensure the time-consistency of the initially chosen imputation.  It can be observed that the players use the same imputation, i.e. the same vector $\alpha$ in (\ref{payment}), throughout the game. However, the optimality principle can include the set of imputations, and it is natural to think about a possible switch from the imputation chosen at the start of the game to another from the same optimality principle at any intermediate time. If the optimality principle is strongly time-consistent, then the sum of the accumulated payoffs up to the moment of switching and the new imputation in the subgame belongs to the initial optimality principle.  This property of the cooperative solution was introduced by L. Petrosyan \citep{Petrosyan93}.

\begin{myDef}\label{def4}
    An optimality principle $P(z_0,0) \subset L(z_0,0)$ is called strongly time-consistent if
    
    1. $P(z^{\pi_1}(t),t)\neq \emptyset,\ \forall t\in[0,\infty]$
    
    2. For any $\zeta\in P(z_0,0)$, there exists an IDP  $\ w(s)=(w_1(s),\ldots,w_n(s)),\ s\in[0,\infty]$, such that $\zeta_i=\int_{0}^{\infty} e^{-\rho s} w_i(s) \ ds   ,i\in N$ and 
    $$P(z_0,0) \supset \int_{0}^{t} e^{-\rho s} w(s) \ ds  \oplus e^{-\rho t}P(z^{\pi_1}(t),t),\forall t\in [0,\infty].$$

\end{myDef}

 Note that  the use of IDP (\ref{tcIDP}) does not guarantee a strong time-consistency of   the optimality principle  $Z(z^{\pi_1}(t),t)$. To demonstrate it, consider an imputation $\zeta^{'}(z^{\pi_1}(\epsilon))\in Z(z^{\pi_1}(\epsilon),\epsilon)$:
\begin{equation}\label{payment'}
\begin{aligned}
\zeta_1^{'}(z^{\pi_1}(\epsilon))&=K_1^{\pi_{42}}(z^{\pi_1}(\epsilon),v^{\pi_{42}})+\alpha_1 ^{'}SC(\epsilon),\\
\zeta_2^{'}(z^{\pi_1}(\epsilon))&=K_2^{\pi_{41}}(z^{\pi_1}(\epsilon),v^{\pi_{41}}) +\alpha_2^{'} SC(\epsilon),\\
\zeta_3^{'}(z^{\pi_1}(\epsilon))&=K_3^{\pi_{3}}(z^{\pi_1}(\epsilon),v^{\pi_{3}}) +\alpha_3^{'} SC(\epsilon).\\
\end{aligned}
\end{equation}

Suppose that up to the moment $\epsilon$ the cooperative payoff is distributed according to IDP (\ref{tcIDP}). The  player 1 will have to this moment 
$$\int_{0}^{\epsilon} e^{-\rho s} w_1(s) \ ds  =K_1^{\pi_{42}}(z_0,0)-e^{-\rho \epsilon}K_1^{\pi_{42}}(z^{\pi_1}(\epsilon),\epsilon)+\alpha_1(SC(0)-e^{-\rho \epsilon}SC(\epsilon)). $$

Then, if there is a switch at the moment $\epsilon$ from the imputation $\zeta$ to $\zeta^{'}$, the resulting payoff of player 1  has the following form
\begin{multline}\label{alpha'}
\int_{0}^{\epsilon} e^{-\rho s} w_1(s) \ ds  + e^{-\rho \epsilon}\zeta^{'}_1(z^{\pi_1}(\epsilon))=\\K_1^{\pi_{42}}(z_0,0)+\alpha_1(SC(0)-e^{-\rho \epsilon}SC(\epsilon))+\alpha_1^{'}SC(\epsilon).\end{multline}

Since $Y>0$ and $M(0)=0$, the integrand $e^{-\rho t}[-L^2(t)Y+ M(t)]$ under the integral of $SC(0)$ at the  moment $t=0$ is negative. That is, there exists a time period from $0$ to $t^{'}$ where this integrand under integral is negative. Hence, $SC(0)-e^{-\rho t^{'}}SC(t^{'})<0$ and we can't guarantee that
$\alpha_1(SC(0)-e^{-\rho \epsilon}SC(\epsilon))+\alpha_1^{'}SC(\epsilon)\geq 0$ for any $\alpha$, $\alpha^{'}$ and $\epsilon$. Thus, the combined imputation (\ref{alpha'}) in the whole game may not belong to  $Z(z^{\pi_1}(0),0)$. Therefore, we conclude that the sustainably-cooperative optimality principle  $Z(z^{\pi_1}(t),t)$ is not strong time-consistent.


\section{Numerical illustration}

\begin{figure}[bpht]
\centering
\includegraphics[scale=0.05]{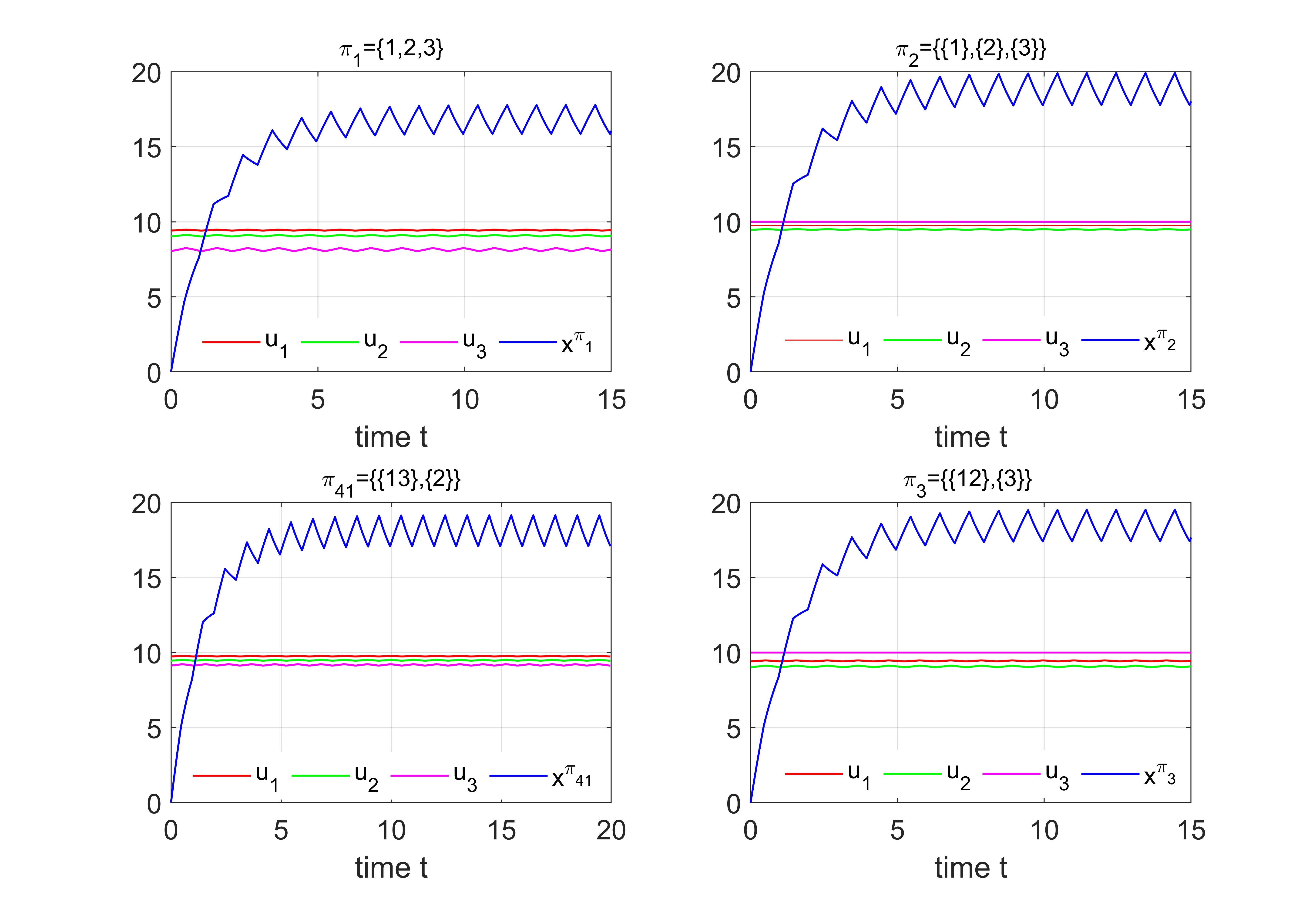}
\caption{Strategies and the corresponding state under different coalition structures.}
\label{fig:constate}
\end{figure}
Satisfying Proposition \ref{Proposition1}, the parameter setting considered in the following analysis is: $\delta_1 = 0.45;\ \delta_2 = 0.9;\ T =1;\ \tau = 0.5;\ \rho=0.3;\ q_1=4;\ q_2=5;\ \xi_1=0.3;\ \xi_2=0.4;\ \xi_3=0.6;\  a_1=5;\ a_2=4;\ a_3=3;\ b_1=b_2=b_3=10;\ \alpha_1=\alpha_2=\alpha_3=1/3;\ z_0=0$. Players' strategies and corresponding state under different coalition structures are shown in Figure \ref{fig:constate}.

Notice that under cooperation coalition, the stock of pollution sustained in the lowest level. It is interesting to note that, in other cases, players' emissions strategies are more aggressive. Especially for player $3$, whose emission rate takes highest admissible value in $\pi_{2}$ and $\pi_{3}$.  By reason of the players’ more reasonable and moderate emission strategies under long term consideration, the cooperative coalition dominates other coalition structures in both economic and environmental aspects.  As the coalition structure $\pi_{42}$ is similar to $\pi_{41}$, it is unnecessary to illustrate this situation.

In Figure \ref{fig:Payment}, where demonstrated the payment to each player in the subgame starting from $\epsilon \in [0,\infty]$, i.e., from top $\epsilon =0$ to the origin $\epsilon =\infty$. The proposed sustainably-cooperative optimality principle (red region) belongs to the imputation set (light yellow region). The blue solid line, belonging to the sustainably-cooperative optimality principle, represents for the payment to players, defined by \eqref{payment}, for the subgame $\Gamma(z^{\pi_1}(\epsilon),\epsilon)$.

\begin{figure}[bpht]
\centering
\includegraphics[scale=0.07]{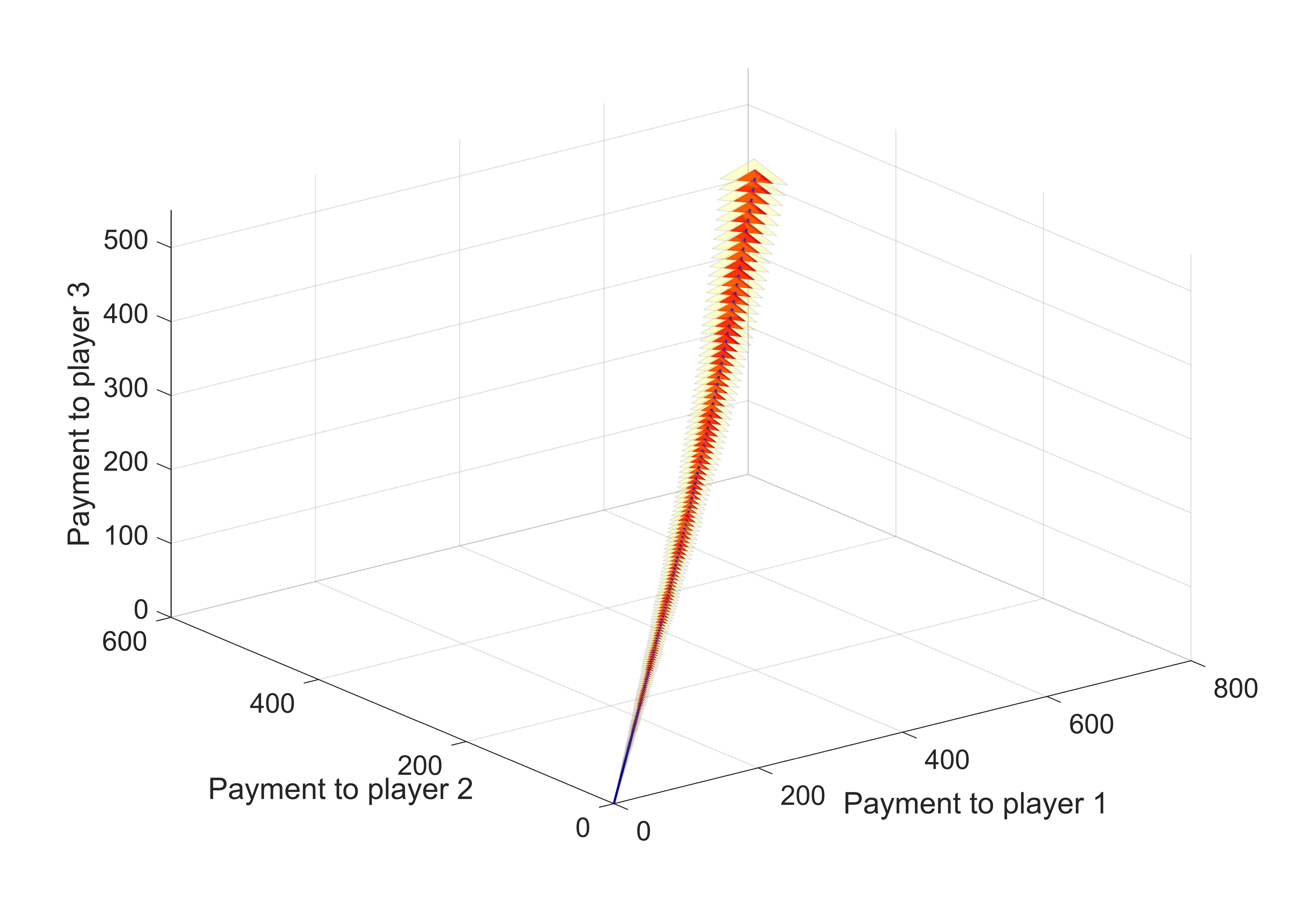}
\caption{Payment to each player.}
\label{fig:Payment}
\end{figure}

\section{Conclusion}

In this paper, we analyzed the 3-player differential game for a typical hybrid optimal pollution-control problem with two types of players: farsighted and myopic. The regular regime shifts are modeled by a piece-wise constant parameter in system dynamics. Considering all possible mechanisms of coalition formation, in each feasible coalition structure, players' strategies are obtained uniquely and the state variable converges to a unique hybrid limit cycle. 

Subsequently, an innovative sustainably-cooperative optimality principle belonging to the imputation is proposed. For this optimality principle, the necessary condition for the existence of time-consistency is formally derived. The imputation distribution procedure (IDP) is established for sustainability of the cooperation agreements. 

The forthcoming study will focus on generalizing the problem formulation to a multi-player game-theoretic framework, characterized by the presence of $n$ heterogeneous players with distinct objectives. 

\section* {CRediT authorship contribution statement}
{\bf Yilun Wu}: Conceptualization; Formal analysis; Investigation; Methodology;Resources; Software; Validation; Visualization; Writing - original draft; Writing - review \& editing.  {\bf Anna Tur}: Conceptualization; Formal analysis; Funding acquisition; Investigation; Methodology; Project administration; Resources; Supervision; Validation; Writing - original draft; Writing - review \& editing.   {\bf Peichen Ye}: Conceptualization; Investigation;Validation; Visualization; Writing - original draft; Writing - review \& editing.

\section* {Declaration of competing interest}
The authors declare that they have no known competing financial interests or personal relationships that could have appeared
to influence the work reported in this paper.

\section* {Data availability}

No data was used for the research described in the article.

\section* {Acknowledgement}
The work of the second author  was supported by the Russian Science Foundation under grant No. 24-21-00302.

\bibliographystyle{apalike}

\end{document}